\documentclass[10pt]{amsart}
\usepackage{amssymb}
\usepackage{amsmath}
\usepackage{amsfonts}

\newtheorem{theorem}{Theorem}[section]
\newtheorem{lemma}[theorem]{Lemma}
\newtheorem{corollary}[theorem]{Corollary}
\newtheorem{propos}[theorem]{Proposition}
\theoremstyle{definition}
\newtheorem{definition}[theorem]{Definition}

\theoremstyle{remark}
\newtheorem{remark}[theorem]{Remark}
\numberwithin{equation}{section}

\newcommand{\semidirect}{\rtimes}

\newfont{\bb}{msbm10}

\def\Bbb#1{\mbox{\bb #1}}

\def\Bbb#1{\mbox{\bb #1}}

\def\N{ {\Bbb N}}

\begin{document}

%%%%%%%%%%%%%%%%%%%%%%%%%%%%%%%%%%%%%%%%%%%%%%%%%%%%%%%%%

\title{Stability of a functional equation of Deeba on semigroups}
\author{VALERI\u I A. FA\u IZIEV}
\address{Fa\u iziev:
Tver State Agricultural Academy,
Tver Sakharovo, Russia}
%\curraddr{Staraya Konstantinovka 70, Tver 170019, Russia}
\email{vfaiz@tvcom.ru}
\author{PRASANNA K. SAHOO}
\address{Sahoo: Department of Mathematics,
University of Louisville,
Louisville, Kentucky 40292 USA}
\email{sahoo@louisville.edu}

\begin{abstract}
Let $S$ be a semigroup and $X$ a Banach space. The functional
equation $\varphi (xyz)+ \varphi (x) + \varphi (y) + \varphi (z) =
\varphi (xy) + \varphi (yz) + \varphi (xz)$ is said to be stable
for the pair $(X, S)$ if and only if $f: S\to X$ satisfying $\|
f(xyz)+f(x) + f(y) + f(z) - f(xy)- f(yz)-f(xz)\| \leq \delta $ for
some positive real number $\delta$ and all $x, y, z \in S$, there
is a solution $\varphi : S \to X$ such that $f-\varphi$ is
bounded. In this paper,  among others, we prove the following
results: 1) this functional equation, in general, is not stable on
an arbitrary semigroup; 2) this equation is stable on periodic
semigroups; 3) this equation is stable on abelian semigroups; 4)
any semigroup with left (or right) law of reduction can be
embedded into a semigroup with left (or right) law of reduction
where this equation is stable. The main results of this paper
generalize the works of Jung \cite{jung1}, Kannappan \cite{ka},
and Fechner \cite{fe}.
\end{abstract}

%%%%%%%%%%%%%%%%%%%%%%%%%%%%%%%%%%%%%%%%%%%%%%%%%%%%%%%%%%

\maketitle{}
\markboth{Valeriy A. Fa\u iziev and Prasanna K. Sahoo}
{Stability of a functional equation of Deeba on semigroups}

\keywords{{\bf Keywords and phrases.} Bimorphism,
Drygas functional equation, embedding, free groups, quadratic functional equation, periodic semigroup,
semidirect product,  stability of functional equation, and wreath product.}

\vskip 2mm
\subjclass{{\bf 2000 Mathematics Subject Classification.} Primary 39B82.}

%%%%%%%%%%%%%%%%%%%%%%%%%%%%%%%%%%%%%%%%%%%%%%%%%%%%%%%%%
%%
%%                    Section 1: Introduction
%%
%%%%%%%%%%%%%%%%%%%%%%%%%%%%%%%%%%%%%%%%%%%%%%%%%%%%%%%%%

\section{Introduction}

If $f : \Bbb{V} \to X$ is a function from a normed vector space $\Bbb{V}$ into a Banach space $X$,
and $\| f(x+y) - f(x) - f(y) \| \leq \delta$,
Hyers \cite{hy}, answering a question of Ulam \cite{ul},
proved that there exists an additive function $A : \Bbb{V} \to X$ such that
$\| f(x) - A(x) \| \leq \delta$. Taking this result into account, the additive
Cauchy functional equation is said to stable in the sense of Hyers-Ulam
on $(\Bbb{V}, X)$ if for each function $f : \Bbb{V} \to X$ satisfying the inequality
$\| f(x+y) - f(x) - f(y) \| \leq \delta$ for some $\delta \geq 0$ and for all
$x, y \in \Bbb{V}$ there exists an additive function $A : \Bbb{V} \to X$ such that
$f-A$ is bounded on $\Bbb{V}$. Since then, the stability problems of various
functional equations have been studied by many authors (see the survey paper
\cite{ra3} and references therein).
Among them, Skof \cite{sk} first considered the
Hyers-Ulam stability of the quadratic functional equation
\begin{equation}\label{eq:quad1.0}
f(xy) + f(xy^{-1})= 2 f(x) + 2 f(y)
\end{equation}
where $f$ maps a group $G$ to an abelian group $H$. As usual, each
solution of equation \eqref{eq:quad1.0} is called a quadratic function.
But Skof restricted herself to studying the case where $f$ maps a normed
space to a Banach space. In \cite{ch} Cholewa noticed that the theorem of Skof's
is true if the relevant domain in replaced by an abelian group. The results
of Skof and Cholewa were further generalized by Czerwik \cite{cz1}.
Further works on stability of the quadratic functional equation can be
found in Feny\H{o} \cite{fe},
Czerwik \cite{cz2}, Czerwik and Dlutek \cite{cz3}, \cite{cz4}, Ger \cite{ger}, Jung \cite{jung2},
Jung and Sahoo \cite{jung3}, and Rassias \cite{ra2}.

\medskip
Let $G$ be a group and $X$ and $Y$ be any two arbitrary Banach spaces over reals.
Faiziev and Sahoo \cite{fs1} proved that the quadratic functional equation is stable
for the pair $(G, X)$ if and only if it is stable for the pair $(G, Y)$. In view of this result
it is not important which Banach space is used on the range. Thus one may
consider the stability of the quadratic functional equation on the pair $(G, \Bbb{R} )$.
Faiziev and Sahoo \cite{fs1} proved that quadratic functional equation is not stable on
the pair $(G, \Bbb{R} )$ when $G$ is any arbitrary group. It is well known (see Skof \cite{sk}
and Cholewa \cite{ch}) that the quadratic functional equation is stable on the pair
$(G, \Bbb{R} )$ when $G$ is an abelian group. Thus it is interesting to know on which
noncommutative groups the quadratic functional equation is stable in the sense of Hyers-Ulam.
Faiziev and Sahoo \cite{fs1} proved that quadratic functional equation is stable on
$n$-abelian groups and $T(2, \Bbb{K})$, where $\Bbb{K}$ is a commutative field.
Further they also proved that every group can be embedded into a group in which the
quadratic functional equation is stable. Yang \cite{dy1} proved the stability of quadratic
functional equation on amenable groups. 

\medskip
In an American Mathematical Society meeting,
E. Y. Deeba of the University of Houston
asked to find the general solution of the functional equation
\begin{equation}\label{ch21:eq:quadvar1.1}
f(x+y+z) + f(x) + f(y) + f(z) = f(x+y) + f(y+z) + f(z+x) .
\end{equation}
This functional equation is a variation of the quadratic functional equation.
Kannappan \cite{ka} showed that the general solution $f : \Bbb{V} \to \Bbb{K}$
of the above functional equation is of the form
$$f(x) = B(x, x) + A(x) $$
where $B : \Bbb{V} \times \Bbb{V} \to \Bbb{K}$ is a symmetric biadditive
function and $A : \Bbb{V} \to \Bbb{K}$ is an additive function, $\Bbb{V}$
is a vector space, and $\Bbb{K}$ is a field of characteristic different from two
(or of characteristic zero).

\medskip
The Hyers-Ulam stability of the equation \eqref{ch21:eq:quadvar1.1} was
investigated by Jung \cite{jung1}. He proved the following theorem.

 \begin{theorem}
 Suppose $\Bbb{V}$ is a real norm space and $X$ a real Banach space.
 Let $f : \Bbb{V} \to X$ satisfy the inequalities
 \begin{equation}
 \| f(x+y+z) + f(x) + f(y) + f(z) - f(x+y) - f(y+z) - f(z+x) \| \leq \delta
 \end{equation}
 and
 \begin{equation}
 \| f(x) - f(- x)  \| \leq \theta
 \end{equation}
 for some $\delta , \theta \geq 0$ and for all $x, y, z \in \Bbb{V}$. Then there exists a
 unique quadratic mapping $Q : \Bbb{V} \to X$ which satisfies
 \begin{equation}
 \| f(x) - Q( x)  \| \leq 3\, \delta
 \end{equation}
 for all $x \in \Bbb{V}$. If, moreover, $f$ is measurable or $f(tx)$ is continuous in $t$
 for each fixed $x \in \Bbb{V}$, then $Q(tx) = t^2 \, Q(x)$ for all $x \in \Bbb{V}$ and $t \in \Bbb{R}$.
 \end{theorem}

Jung \cite{jung1} proved another theorem replacing the inequality
$ \| f(x) - f(- x)  \| \leq \theta $
by  $\| f(x) + f(- x)  \| \leq \theta$. 
Fechner \cite{wf} proved the stability of the functional equation \eqref{ch21:eq:quadvar1.1}
on abelian group. 
For this functional equation \eqref{ch21:eq:quadvar1.1},
Kim \cite{ghk} proved a generalized stability result in the spirit of Gavruta \cite{gav}.
Chang and Kim \cite{ck} generalized the theorem of Jung \cite{jung1}
and proved the following theorem.

\begin{theorem}
 Suppose $\Bbb{V}$ is a real norm space and $X$ a real Banach space.
 Let $H : \Bbb{R}_+^3 \to \Bbb{R}_+$ be a function such that
 $H( tu, tv, tw ) \leq t^p H(u, v, w)$ for all $t, u, v, w \in \Bbb{R}_+$
 and for some $p \in \Bbb{R}$. Further, let $E : \Bbb{R}_+ \to \Bbb{R}_+$
 satisfying $E(tx) \leq t^q E(x)$ for all $t, x \in \Bbb{R}_+$.
 Let $p, q < 1$ be real numbers and let $f : \Bbb{V} \to X$ satisfy the inequalities
 \begin{align}\nonumber
 \| f(x&+y+z) + f(x) + f(y) + f(z) \\
 &- \, f(x+y) - f(y+z) - f(z+x) \| \leq H( \|x\|, \|y\|, \|z\| )
 \end{align}
 and
 \begin{equation}
 \| f(x) - f(- x)  \| \leq E( \| x\| )
 \end{equation}
 for some $\delta , \theta \geq 0$ and for all $x, y, z \in \Bbb{V}$. Then there exists a
 unique quadratic mapping $Q : \Bbb{V} \to X$ which satisfies
 \begin{equation}
 \| f(x) - Q( x)  \| \leq {{ H( \|x\|, \|x\|, \|x\| )} \over {2-2^p}} + 2 \| f(0) \|
 \end{equation}
 for all $x \in \Bbb{V}$. If, moreover, $f$ is measurable or $f(tx)$ is continuous in $t$
 for each fixed $x \in \Bbb{V}$, then $Q(tx) = t^2 \, Q(x)$ for all $x \in \Bbb{V}$ and $t \in \Bbb{R}$.
 \end{theorem}

\medskip
Chang and Kim \cite{ck} also proved another similar theorem
replacing the inequality $\| f(x) - f(- x)  \| \leq E( \| x\| )$
by  $\| f(x) + f(- x)  \| \leq E( \| x\| )$.

\medskip
The functional equation \eqref{ch21:eq:quadvar1.1} is takes the form
\begin{equation}\label{ch21:eq:quadvar1.2}
f(xyz) + f(x) + f(y) + f(z) = f(xy) + f(yz) + f(xz) 
\end{equation}
on an arbitrary group $G$ or on a semigroup $S$.
In this sequel, we will
write the arbitrary semigroup $S$ in multiplicative notation.
Similarly, the arbitrary group $G$ will be written in multiplicative notation
so that $1$ will denote the identity element of $G$.
This functional equation implies the Drygas
functional equation $f(xy) + f(xy^{-1} ) = 2 f(x) + f(y) + f(y^{-1} )$ whose general
solution was presented in Ebanks, Kannappan and Sahoo \cite{eks}.
The stability
of the Drygas functional equation was studied by Jung and Sahoo \cite{jung3} and also by
Yang \cite{dy}. The system of equations
$f(xy) + f(xy^{-1} ) = 2 f(x) + f(y) + f(y^{-1} )$ and $f(yx) + f(y^{-1} x) = 2 f(x) + f(y) + f(y^{-1} )$
generalizes the Drygas functional equation on groups. The stability of this system of equation was
investigated by Faiziev and Sahoo (see \cite{fs2}, \cite{fs3}, and \cite{fs4}) on
nonabelian groups.

\medskip
In the present paper, we consider the stability of the
functional equation~\eqref{ch21:eq:quadvar1.2} for the pair $(S,E)$ when
$S$ is an arbitrary semigroup and $E$ is a real Banach space. If $X$
is another real Banach space, then we prove that
the functional equation~\eqref{ch21:eq:quadvar1.2} is stable for the pair
$(S, X)$ if and only if it is stable for the pair $(S, E)$. We show that,  in general, the
equation \eqref{ch21:eq:quadvar1.2} is not stable on semigroups.
However, this equation \eqref{ch21:eq:quadvar1.2} is stable on
periodic semigroups as well as abelian semigroups.
We also show that any semigroup with left (or right) cancellation law
can be embedded into a semigroup with left (or right) cancellation law
where the equation  \eqref{ch21:eq:quadvar1.2} is stable.
The main results of this paper
generalize the works of Jung \cite{jung1}, Kannappan \cite{ka},
and Fechner \cite{fe}.

%%%%%%%%%%%%%%%%%%%%%%%%%%%%%%%%%%%%%%%%%%%%%%%%%%%%%%%%%
%%
%%                    Section 2: Decomposition
%%
%%%%%%%%%%%%%%%%%%%%%%%%%%%%%%%%%%%%%%%%%%%%%%%%%%%%%%%%%

\section{Decomposition}

Let $S$ be a semigroup and $X$ be a Banach space. Let $\Bbb{N}$ be the set of natural
numbers and $\Bbb{Z}$ be the set of integers. Moreover, let $\Bbb{R}$ denote the set of 
real numbers.

\begin{definition} \label{definition2.0}
A mapping
$f: S \to X$ is said to be {\it a kannappan}\;\;mapping if it satisfies equation
\begin{equation}\label{kan}
f(xyz)+f(x)+f(y)+f(z)-f(xy)-f(xz)-f(yz)=0.
\end{equation}
\end{definition}

\begin{definition} \label{definition2.1}
We will say that $f:S\to X$ is a quasikannappan mapping if there
is $c>0$ such that
\begin{equation}\label{qkan}
\|f(xyz)+f(x)+f(y)+f(z)-f(xy)-f(xz)-f(yz)\|\le c
\end{equation}
for all $ x,y,z\in S$.
\end{definition}

The set of kannappan and quasikannappan mappings will be denote by
$K(S, X)$ and $KK(S,X)$, respectively.

\begin{lemma} \label{lemma2.2}
If $f \in KK(S, X)$, then
for any $n\ge 3$ and $x_1,\dots , x_n\in S$ the
inequality
\begin{equation} \label{eq:1.2}
\bigg \|\, f(x_1x_2\cdots x_n) +(n-2)\sum_{i=1}^nf(x_i)-\sum_{1\le
i<j\le n }f(x_ix_j) \,\bigg \|\le \frac{(n-2)(n-1)}{2} \, c
\end{equation}
holds.
\end{lemma}

\begin{proof}
We prove this lemma by induction. First we show that the inequality \eqref{eq:1.2}
is true for $n=4$. Since $f \in KK(S, X)$, we obtain from \eqref{qkan}
\begin{eqnarray*}
\|\, f(x_1x_2x_3x_4) + f(x_1)+f(x_2)+f(x_3x_4) \qquad\qquad\qquad \quad \\
 - \, f(x_1x_2)-f(x_1x_3x_4)-f(x_2x_3x_4) \,\|\le c ,
\end{eqnarray*}
\begin{eqnarray*}
\|\, f(x_2x_3x_4) +
f(x_2)+f(x_3)+f(x_4)-f(x_2x_3)-f(x_2x_4)-f(x_3x_4) \,\|\le c,
\end{eqnarray*}
and
\begin{eqnarray*}
\|\, f(x_1x_3x_4) +
f(x_1)+f(x_3)+f(x_4)-f(x_1x_3)-f(x_1x_4)-f(x_3x_4) \,\|\le c .
\end{eqnarray*}
Therefore from the above three inequalities we have
\begin{eqnarray*}
\|\, f(x_1x_2x_3x_4) + f(x_1)+f(x_2)+f(x_3x_4)-f(x_1x_2) \qquad\qquad\qquad \quad
\\
+f(x_1)+f(x_3)+f(x_4)-f(x_1x_3)-f(x_1x_4)-f(x_3x_4) \qquad\quad
\\
+f(x_2)+f(x_3)+f(x_4)-f(x_2x_3)-f(x_2x_4)-f(x_3x_4)  \,\|\le 3c .
\end{eqnarray*}
Simplifying we see that
\begin{eqnarray*}
\begin{array}{l}
\|\, f(x_1x_2x_3x_4)+2[f(x_1)+f(x_2)+f(x_3)+f(x_4)] -f(x_1x_2) \qquad  \\
\qquad \quad - \, f(x_1x_3) -f(x_1x_4)-f(x_2x_3)-f(x_2x_4)-f(x_3x_4)
\,\|\le 3c
\end{array}
\end{eqnarray*}
and this shows that inequality \eqref{eq:1.2} holds for $n=4$. We will rewrite the
above inequality as
\begin{eqnarray*}
\begin{array}{l}
\|\, f(x_1x_2x_3x_4)+2[f(x_1)+f(x_2)+f(x_3)+f(x_4)] -f(x_1x_2) \qquad  \\
\qquad \quad - \, f(x_1x_3) -f(x_1x_4)-f(x_2x_3)-f(x_2x_4)-f(x_3x_4)
\,\|\le c_4
\end{array}
\end{eqnarray*}
where $c_4 = 3 \, c$.
Next suppose the above inequality
holds for a positive integer $n$. That is
\begin{eqnarray*}
\bigg \|\, f(x_1x_2\cdots x_n) +(n-2)\sum_{i=1}^nf(x_i)-\sum_{1\le
i<j\le n }f(x_ix_j) \, \bigg \|\le c_n .
\end{eqnarray*}
Consider
\begin{eqnarray*}
\bigg \|\, f(x_1x_2\cdots x_nx_{n+1})
+(n-1)\sum_{i=1}^{n+1}f(x_i)-\sum_{1\le i<j\le n+1 }f(x_ix_j)
\, \bigg \| .
\end{eqnarray*}
By our supposition we have
\begin{eqnarray*}
\begin{array}{l}
\bigg \|\, f(x_1x_2\cdots (x_nx_{n+1}))
+(n-2) \left [ \displaystyle{\sum_{i=1}^{n-1}} f(x_i)+f(x_nx_{n+1}) \right ]
\qquad \qquad  \\
\qquad\qquad \qquad \qquad  -\displaystyle{\sum_{1\le i<j\le
n-1}} f(x_ix_j) - \displaystyle{\sum_{1\le i\le n -1}} f(x_ix_nx_{n+1})\, \bigg \|\le c_n .
\end{array}
\end{eqnarray*}
Hence
\begin{eqnarray*}
\begin{array}{l}
\bigg \|\, f(x_1x_2\cdots (x_nx_{n+1}))
+ (n-2) \left [ \displaystyle{\sum_{i=1}^{n-1}}f(x_i)+f(x_nx_{n+1}) \right ]
 -\displaystyle{\sum_{1\le i<j \le n-1} }f(x_ix_j) \\
\qquad  + \displaystyle{\sum_{1\le i\le n -1}}
\left [f(x_i)+f(x_n)+f(x_{n+1})-f(x_ix_n)-f(x_ix_{n+1})-f(x_nx_{n+1}) \right ]\, \bigg \| \\
\\
\le c_n + (n-1) \, c.
\end{array}
\end{eqnarray*}
The last inequality can be rewritten as
\begin{eqnarray*}
\begin{array}{l}
\bigg \|\, f(x_1x_2\cdots (x_nx_{n+1}))
+ (n-1) \displaystyle{\sum_{i=1}^{n+1}} f(x_i) + (n-2) f(x_nx_{n+1}) \\
\qquad \quad
 - \displaystyle{\sum_{1\le i<j\le n-1}} f(x_ix_j)
 + \displaystyle{\sum_{1\le i\le n -1}}
\left [\,  - \, f(x_ix_n)-f(x_ix_{n+1})-f(x_nx_{n+1}) \, \right ]\, \bigg \| \\ \\
\le c_n+(n-1)\, c.
\end{array}
\end{eqnarray*}
Hence
\begin{eqnarray*}
\begin{array}{l}
\bigg \|\, f(x_1x_2\cdots x_nx_{n+1}) +(n-1)\displaystyle{\sum_{i=1}^{n+1}} f(x_i )
 -\displaystyle{\sum_{1\le i<j\le n+1}} f(x_ix_j) \, \bigg \|
 \le c_{n+1}
\end{array}
\end{eqnarray*}
where $c_{n+1} = c_n + (n-1) \, c $ for $n \geq 3$.

\medskip
From the recurrence relations $c_3 = c$ and $c_{n+1} = c_n + (n-1) \, c $ for $n \geq 3$,
we get
$$ c_{n+1}=\frac{n\, (n-1)}{2}.$$
Thus we have proved the inequality \eqref{eq:1.2} for all positive integers $n$.
\end{proof}

The following lemma follows from the above lemma.

\begin{lemma}\label{lemma2.3}
If $f \in KK(S, X)$, then
for any $n\ge 3$, the inequality
\begin{eqnarray}
\label{22s}
\qquad \bigg \|\, f(x^n) + (n-2)\, n \, f(x)- \frac{(n-1)\, n}{2}f(x^2)
\,\bigg \|
\le \frac{(n-2)\, (n-1)}{2} \, c
\end{eqnarray}
holds for all $x\in S$.
\end{lemma}

\begin{proof}
Letting $x_1 = x_2 = \cdots = x_n = x$ in the inequality \eqref{eq:1.2}, we have the
asserted inequality \eqref{22s}.
\end{proof}

\begin{lemma} \label{f(x^2)} %%% \label{lemma2.4}
Let the function $\phi : S \to X$ be define by $\phi(x)=f(x^2)$.
\begin{enumerate}
\item
If $f\in KK(S,X)$, then $\phi\in KK(S,X)$.
\item
If $f\in K(S,X)$, then $\phi\in K(S,X)$.
\end{enumerate}
\end{lemma}

\begin{proof}
Since $f\in KK(S,X)$, we have
$$ \|\, f(xyz)+ f(x)+f(y)+f(z)-f(xy)-f(xz)-f(yz)\,\|\le c. $$
Consider
\begin{align*}
\|\, \phi(xyz) &+ \phi(x) + \phi(y) + \phi(z) - \phi(xy) - \phi(xz) - \phi(yz)\,\|   \\
&= \, \|\, f(xyzxyz) + f(xx) + f(yy) + f(zz) -  f(xyxy) - f(xzxz) - f(yzyz)\,\| .
\end{align*}
We have
\begin{align} \label{www0}
\|\,f((xy)z(xy)z)+4f(xy)&+4f(z) -f(xyz)-f(xyxy)\\ \nonumber
&-f(xyz) - \, f(zxy)-f(zz)-f(xyz) \,\|\le 3c ,
\end{align}
\begin{equation} \label{www1}
\|\,f((xy)z(xy)z)+4f(xy)+4f(z) -3f(xyz)-f(xyxy)-f(zxy)-f(zz) \,\|\le 3c ,
\end{equation}
\begin{equation} \label{www2}
\|\,f(xzxz)+4f(x)+4f(z)  -3f(xz)-f(zx)-f(x^2)-f(z^2) \,\|\le 3c ,
\end{equation}
\begin{equation} \label{www3}
 \|\,f(yzyz)+4f(y)+4f(z) -3f(yz)-f(zy)-f(y^2)-f(z^2) \,\|\le 3c .
\end{equation}
From~\eqref{www0}--\eqref{www3} we have
\begin{align*}
\|\, f(xyzxyz) &+ f(xx)+f(yy)+f(zz)-f(xyxy)-f(xzxz)-f(yzyz)\,\| \\
& = \;  \|\, f(xyzxyz)+  4f(xy)+4f(z) -3f(xyz)-f(xyxy)-f(zxy) \\
&\qquad - f(zz) -4f(xy)-4f(z) +3f(xyz)+f(xyxy)+f(zxy) \\
& \qquad +f(zz) + f(xx)+f(yy)+f(zz)-f(xyxy) \\
& \qquad - f(xzxz)-4f(x)-4f(z)+3f(xz)+f(x^2)+f(zx)+f(z^2) \\
&\qquad + 4f(x)+4f(z)-3f(xz)-f(x^2)-f(zx)-f(z^2) \\
& \qquad - f(yzyz)-4f(y)-4f(z)+3f(yz)+f(y^2)+f(zy)+f(z^2) \\
&\qquad +4f(y)+4f(z)-3f(yz)-f(y^2)-f(zy)-f(z^2) \,\| .
\end{align*}
Therefore
\begin{eqnarray*}
\begin{array}{l}
\|\, f(xyzxyz)+ f(xx)+f(yy)+f(zz)-f(xyxy)-f(xzxz)-f(yzyz)\,\| \\
\qquad \quad \le \; \|\, f(xyzxyz)+  4f(xy)+4f(z) -3f(xyz)-f(xyxy)-f(zxy)-f(zz)\| \\
\qquad \qquad + \, \|-f(xzxz)-4f(x)-4f(z)+3f(xz)+f(x^2)+f(zx)+f(z^2)\|  \\
\qquad \qquad + \, \|-f(yzyz)-4f(y)-4f(z)+3f(yz)+f(y^2)+f(zy)+f(z^2)\| \\
\qquad \qquad + \, \|-4f(xy)-4f(z) +3f(xyz)+f(xyxy)+f(zxy) \\
\qquad \qquad \qquad \qquad \qquad + \, f(zz) + f(xx)+f(yy)+f(zz) -f(xyxy) \\
\qquad \qquad \qquad \qquad \qquad + \, 4f(x)+4f(z)-3f(xz)-f(x^2)-f(zx)-f(z^2) \\
\qquad \qquad \qquad \qquad \qquad + \, 4f(y)+4f(z)-3f(yz)-f(y^2)-f(zy)-f(z^2) \,\| .
\end{array}
\end{eqnarray*}
Notice that
\begin{eqnarray*}
\begin{array}{l}
\|-4f(xy)-4f(z) +3f(xyz)+f(xyxy)+f(zxy) +f(zz) + f(xx) \\
\qquad + \, f(yy)+f(zz) -f(xyxy) + 4f(x)+4f(z)-3f(xz) \\
\qquad  - \, f(x^2)-f(zx)-f(z^2) +4f(y)+4f(z) \\
\qquad - \, 3f(yz)-f(y^2)-f(zy)-f(z^2)  \,\;\| \\
= \|\;3f(xyz)+f(zxy)+4f(z)+4f(x)+4f(y) -4f(xy)-3f(xz)-3f(yz) \\
\qquad \quad-f(zx)-f(zy) \,\;\| \\
\le \|\;f(zxy)+f(z)+f(x)+f(y) -f(xy) -f(zx)-f(zy)  \,\;\|  \\
\qquad + \|\;3f(xyz)+3f(z)+3f(x)+3f(y) -3f(xy)-3f(xz)-3f(yz) \,\;\| \\
 \le 3c+9c=12c.
 \end{array}
\end{eqnarray*}
Hence
\begin{align*}
\|\, f(xyzxyz)+ f(xx)+f(yy)+f(zz)-f(xyxy)-f(xzxz)-f(yzyz)\,\|  \\
\le 3c+3c+3c+12c=21c.
\qquad\qquad\qquad\qquad \qquad\qquad\qquad\qquad \qquad\quad
\end{align*}
Thus from the last inequality we have
\begin{eqnarray*}
\|\, \phi(xyz)+
\phi(x)+\phi(y)+\phi(z)-\phi(xy)-\phi(xz)-\phi(yz)\,\| \le 21c.
\end{eqnarray*}
The proof of (2) follows similarly.
\end{proof}

\begin{lemma} \label{lemma2.5}
Let $\big\{ a_k \big\}_1^\infty$ be a sequence in $X$ such that
for any $m, k \in \Bbb{N}$
\begin{eqnarray}
\label{a_k}
 \|\, a_{m+k} -2a_{k+1}+a_k\,\|\le \frac{d}{4^k}
\end{eqnarray}
holds. Then $\big\{ a_k \big\}_1^\infty$ is a Cauchy sequence.
\end{lemma}
\begin{proof}
For any positive integers $n, m$ we have
\begin{equation*}
\label{a_k0}
\|\, a_{m+k} -2a_{k+1}+a_k\,\|\le \frac{d}{4^k},
\end{equation*}
and
\begin{equation*}
\|\, a_{n+k} -2a_{k+1}+a_k\,\|\le \frac{d}{4^k}.
\end{equation*}
Hence
\begin{eqnarray*}
\label{a_k1}
 \|\, a_{n+k} -a_{m+k}\,\|\le \frac{2d}{4^k}.
\end{eqnarray*}
The latter inequality implies that $\big\{ a_k \big\}_1^\infty$
is a Cauchy sequence.
\end{proof}

\begin{lemma} \label{lemma2.6}
Suppose $f \in KK (S, X)$.
For any $x\in S$,  the limit
\begin{equation}\label{l-7-1}
\lim_{n\to \infty }\frac{1}{4^n}f(x^{2^n})=\widehat f(x),
\end{equation}
exists and it satisfies the relations
\begin{equation}\label{l-7-2}
\widehat f(x^n)=n^2 \, \widehat f(x),
\end{equation}
\begin{equation}\label{l-7-3}
\bigg \| \, \widehat f(x)- \bigg [ \, \frac{1}{2}f(x^2)- f(x) \, \bigg ] \, \bigg \|\le \frac{1}{2} \, c
\end{equation}
for all $x \in S$ and $n \in \Bbb{N}$.
\end{lemma}
\begin{proof}
From~(\ref{22s}) it follows that
\begin{equation} \label{22s1}
\bigg \|\, \frac{1}{n^2} \, f(x^n) + \left (1-\frac{2}{n} \right ) f(x)-
\frac{1}{2} \left (1-\frac{1}{n} \right ) f(x^2) \,\bigg \|
\le \frac{1}{2} \left (1- \frac{3}{n}+\frac{2}{n^2} \right ) c,
\end{equation}
\begin{equation}
\\
\label{22s1a}
 \bigg \|\, \frac{1}{n^2} \, f(x^n) - \bigg [\frac{1}{2}f(x^2) -
f(x) \bigg ] + \frac{1}{2n} f(x^2) -\frac{2}{n} f(x)\bigg \| \le
\frac{1}{2} c.
\end{equation}
Therefore, there is $n_0$ such that if $n>n_0 $, then
\begin{equation} \label{22s3}
\bigg \|\, \frac{1}{n^2}f(x^n) - \bigg [\frac{1}{2}f(x^2) - f(x)
\bigg ]\, \bigg \|\le  c.
\end{equation}
Therefore, in \eqref{22s1} and \eqref{22s3} replacing $n$ by $2^m$, we have
\begin{eqnarray}
\label{22s4} \bigg \|\, \frac{1}{4^m}\, f(x^{2^m}) - \bigg [ \,
\frac{1}{2}\, f(x^2) - f(x) \, \bigg ]\, \bigg \|\le  c,
\end{eqnarray}
and
\begin{eqnarray}
\label{22s5} \bigg \|\, \frac{1}{4^{m+k}}\, f(x^{2^{m+k}}) - \bigg
[\frac{2}{4^{k+1}}\, f(x^{2^{k+1}}) - \frac{1}{4^k}\, f(x^{2^k})
\bigg ]\, \bigg \|\le \frac{1}{4^k} \, c.
\end{eqnarray}
Denote $\frac{1}{4^k}\, f(x^{2^k})$ by $a_k$. Then
from~$(\ref{22s5})$ we have
\begin{eqnarray}
\label{22s6}
 \|\, a_{m+k} -2a_{k+1}+a_k\,\|\le \frac{1}{4^k} \, c.
\end{eqnarray}
Now from Lemma  \ref{lemma2.5}  it follows that the sequence
$\bigg\{\,a_k :=\frac{1}{4^k}f \left (x^{2^k} \right )\,\bigg\}$
is a Cauchy sequence and thus has a limit. This
limit we denote by $f_2(x)$. So
\begin{eqnarray}
\label{f_2} f_2(x)=\lim_{k\to \infty }\frac{1}{4^k}\, f \left (x^{2^k} \right ) .
\end{eqnarray}
Hence we have
\begin{eqnarray*}
\label{f_2-m} f_2\left (x^{2^m} \right )&=&\lim_{k\to \infty}\frac{1}{4^k}\, f \left (x^{2^{k+m}} \right ) \\
&=& \lim_{k\to \infty }\frac{4^m}{4^{k+m}}\, f\left (x^{2^{k+m}} \right ) \\
&=& 4^m\lim_{k\to \infty }\frac{1}{4^{k+m}}\, f \left (x^{2^{k+m}} \right ) \\
&=& 4^m\, f_2(x).
\end{eqnarray*}
From the relation~$(\ref{22s4})$ it follows that
\begin{eqnarray}
\label{f_2s4} \bigg \|\, f_2(x) - \bigg [ \frac{1}{2}f(x^2) - f(x)
\bigg ]\, \bigg \| \le c.
\end{eqnarray}
Taking into account Lemma \ref{f(x^2)} we see that $f_2\in KK(S,X)$.

\medskip
Now let $m\ge 3$ be a positive integer. Then for any $x\in S$ we have
\begin{equation*}
\bigg  \|\, f_2 \left (x^{m^n} \right ) +(m^n-2)\, m^n\, f_2(x)- \frac{(m^n-1)\, m^n}{2}f_2 \left (x^2 \right )
\,\bigg \|\le \frac{(m^n-2)\, (m^n-1)}{2}\, c.
\end{equation*}
Dividing the both sides of the  last inequality by $m^{2n}$ and simplifying, we have
\begin{equation*}
 \bigg  \|\, \frac{1}{m^{2n}}\, f_2\left (x^{m^n} \right ) +\left (1- \frac{2}{m^n}\right ) f_2(x)
 - \left (\frac{1}{2}- \frac{1}{2m^n} \right )f_2 \left (x^2 \right ) \, \bigg \|
 \le \frac{(m^n-2)\, (m^n-1)}{2\, m^{2n}}\, c.
 \end{equation*}
 Hence we have
\begin{equation*}
\bigg  \|\, \frac{1}{m^{2n}}\, f_2 \left (x^{m^n} \right ) - \bigg   [\frac{1}{2}\, f_2(x^2) -f_2(x) \bigg  ]
  - \frac{2}{m^n}\, f_2(x)+
 \frac{1}{2m^n}\, f_2(x^2)
 \, \bigg  \|\le \frac{1}{2}\, c.
\end{equation*}
Therefore, there is a $n_0$ such that if $n\ge n_0$, then
\begin{equation*}
\bigg  \|\, \frac{1}{m^{2n}}\, f_2\left (x^{m^n} \right ) - \bigg  [\frac{1}{2}f_2 \left (x^2 \right ) -f_2(x) \bigg  ]
  \, \bigg  \|\le  c.
\end{equation*}
From the later relation it follows that
\begin{equation*}
\bigg \|\, \frac{1}{m^{2n}}\, f_2 \left (x^{m^{n+k}} \right ) -
\bigg [\frac{1}{2}\, f_2 \left ( \big (x^{m^k} \big )^2 \right )
-f_2 \left (x^{m^k} \right ) \bigg ] \,\bigg \|\le c.
\end{equation*}
Now dividing the both sides of the last inequality by $m^{2k}$, we obtain
\begin{equation*}
\bigg \|\, \frac{1}{m^{2(n+k)}}\, f_2 \left (x^{m^{n+k}} \right )
- \frac{1}{m^{2k}}\bigg [\frac{4}{2}\, f_2 \left (x^{m^k} \right )
-f_2 \left (x^{m^k} \right ) \bigg ] \, \bigg \|\le
\frac{1}{m^{2k}}\, c
 \end{equation*}
and thus
\begin{equation*}
\bigg  \|\, \frac{1}{m^{2(n+k)}}\, f_2 \left (x^{m^{n+k}} \right )
- \frac{1}{m^{2k}}\, f_2 \left (x^{m^k} \right )  \, \bigg \| \le
\frac{1}{m^{2k}}\, c.
 \end{equation*}
From the last relation it follows that there is a limit
\begin{equation}
f_m(x) = \lim_{n\to \infty}\frac{1}{m^{2k}} \, f_2 \left (x^{m^k} \right ).
\end{equation}
It is clear that for any $q\in \Bbb{N}$ and $x\in S$ the following relations hold:
\begin{equation}
f_m\left (x^{m^q} \right ) = m^{2q} \, f_m (x),\qquad f_m \left (x^{2^q} \right )=4^{q} \, f_m(x).
\end{equation}
Moreover we have
 \begin{eqnarray*}
\bigg \|\;f_m(x)- \bigg [\frac{1}{2}f_2(x^2)-f_2(x) \bigg ] \;
\bigg \|\le c.
\end{eqnarray*}
Taking into account  relation $f_2(x^{2^k})=4^kf_2(x)$ we get
 \begin{eqnarray*}
 \|\;f_m(x)-f_2(x) \;\|\le 2c.
\end{eqnarray*}
Now taking into account  relation $f_m(x^{2^k})=4^kf_m(x)$ we get
 \begin{eqnarray*}
 f_m(x)=f_2(x) \quad \forall x\in S.
\end{eqnarray*}
Now if we  denote $f_2(x)$ by $\widehat f(x)$ we obtain $\widehat
f(x^n)=n^2\widehat f(x)$ and the proof of the lemma is now complete.
\end{proof}

%\begin{remark}\label{remark2.7}
\begin{corollary}\label{corollary2.7}
If $f \in K(S, X)$, then the limit $\widehat f(x) =
\displaystyle{\lim_{n\to \infty }}\frac{1}{4^n}f(x^{2^n})$ exists
and satisfies $\widehat f(x^n)=n^2 \, \widehat f(x) $ for all $x
\in S$ and $n \in \Bbb{N}$. Moreover, $\widehat f(x)\in K(S,X)$
and $\widehat f(x)=\frac{1}{2}f(x^2)-f(x)$.
%\end{remark}
\end{corollary}
\begin{lemma} \label{widetildef}  %%% \label{lemma2.7}
Let the function $f:S\to X$ satisfy the condition
\begin{equation}
\|\,f(x^2)-2f(x) \,\|\le c
\end{equation}
for some $c > 0$ and all $x \in S$. Then there is a limit
\begin{equation*}
\widetilde f(x) = \lim_{k\to \infty} \frac{1}{2^k} \, f \left (x^{2^k} \right ) ,
\end{equation*}
and for any $m \in \Bbb{N}$ and $x\in S$ the following relations
\begin{eqnarray}
\widetilde f(x^m)=m \, \widetilde f(x), \\
\|\,\widetilde f(x) - f(x)\,\|\leq c.
\end{eqnarray}
hold.
\end{lemma}

\begin{proof}
The proof is similar to the proof of the previous lemma.
\end{proof}

\begin{lemma} \label{f-widetatf}   %%%  \label{lemma2.9}
For any $f\in KK(S,X)$, the function $\varphi = f -\widehat f$
satisfies inequality
 \begin{eqnarray}
 \label{f-widetatf-1}
\|\; \varphi(x^2)-2\varphi(x)\;\|\le c
\end{eqnarray}
for some positive $c$ and any $x\in S$.
\end{lemma}

\begin{proof}
Let $f\in KK(S,X)$.
Then $f$ satisfies relation~$(\ref{qkan})$. Hence from
~$(\ref{l-7-3})$, we get
\begin{eqnarray*}
\|\; 2\widehat f(x)-f(x^2)+2f(x)\;\|\le c.
\end{eqnarray*}
Now we obtain
 \begin{align*}
\|\; \varphi(x^2)-2\varphi(x)\;\|
&=\|\; f(x^2)-\widehat f(x^2) - 2f(x)+2\widehat f(x)\;\| \\
&=\|\; f(x^2)-4\widehat f(x) - 2f(x)+2\widehat f(x)\;\| \\
&=\|\; f(x^2) - 2f(x)-2\widehat f(x)\;\| \\
&\le c
\end{align*}
and the proof of the lemma is complete.
\end{proof}
From above lemma and Corollary~\ref{corollary2.7} we get the
following Corollary.
%\begin{remark}\label{remark2.9}
\begin{corollary}\label{corollary2.9}
If $f \in K(S, X)$, then the function defined by $\phi = f -
\widehat{f}$ satisfies $\phi ( x^2 ) = 2 \, \phi (x)$ for all $x
\in S$ and belongs to $K(S,X)$.
%\end{remark}
\end{corollary}
Denote by $PK_4(S,X)$ and $PK_2(S,X)$ the subspaces of $KK(S,X)$
consisting of functions $f$ satisfying
\begin{equation*}
f(x^k) = k^2f(x)\;\;\forall \, k\in \Bbb{N}, \,\, \forall \, x \in S,
\end{equation*}
and
\begin{equation*}
f(x^k) = kf(x)\;\;\forall \, k\in \Bbb{N}, \,\, \forall \, x \in S,
\end{equation*}
respectively.

\begin{theorem}\label{theorem2.10}
\label{decomposition} For any semigroup $S$ we have the following
decomposition:
\begin{eqnarray*}
KK(S,X)=PK_4(S,X)\oplus PK_2(S,X)\oplus B(S,X),
\end{eqnarray*}
where $B(S,X)$ denotes the space of all bounded mappings from $S$ to $X$.
\end{theorem}

\begin{proof}
It is clear that  $KK(S,X)$ is the direct sum of
$PK_4(S,X)$, $PK_2(S,X)$ and  $B(S,X)$. To see this, let
$f$ be a quasikannappan function satisfying
inequality~$(\ref{qkan})$. Then function $\varphi=f-\widehat f$
belongs to $KK(S,X)$ and satisfies
relation \eqref{f-widetatf-1}. Now from Lemma \ref{widetildef} and Lemma
\ref{f-widetatf}
it follows that $\widetilde\varphi\in KK(S,X)$ and
\begin{eqnarray*}
\|\; \widetilde\varphi(x)-\varphi(x) \;\|\le c
\end{eqnarray*}
for any $x\in S$. So, the function $\delta(x)=f(x)-\widehat
f(x)-\widetilde\varphi(x) $ is bounded. We can rewrite the last relation
as $f(x)=\widehat f(x)+ \widetilde\varphi(x)+\delta(x) $
and hence $KK(S,X)=PK_4(S,X)\oplus PK_2(S,X)\oplus B(S,X)$.
\end{proof}

%%%%%%%%%%%%%%%%%%%%%%%%%%%%%%%%%%%%%%%%%%%%%%%%%%%%%%%%%
%%
%%                    Section 3: Stability
%%
%%%%%%%%%%%%%%%%%%%%%%%%%%%%%%%%%%%%%%%%%%%%%%%%%%%%%%%%%

\section{Stability}

\begin{definition}
Let $S$ be a semigroup and $X$ be a Banach space.
The functional equation~$(\ref{kan})$ is said to be stable for the pair $(S,X)$
if for any $f: S\to X$  satisfying inequality
\begin{equation} \label{st}
\|\; f(xyz)+f(x)+f(y)+f(z)-f(xy)-f(xz)-f(yz)\;\|\le d
\end{equation}
for some positive real number $d$ and all $x,y,z\in S$, then there is a solution
$\varphi$ of~$(\ref{kan})$ such that the difference $f-\varphi$ is
a bounded mapping.
\end{definition}

The subspace of $K(S, X)$ consisting of functions belonging to $PK_4 (S, X)$
will be denoted by $K_4(S, X)$. In other words $K_4(S, X)$ consists of
solutions of \eqref{kan} satisfying the additional condition
\begin{equation*}
f(x^k) = k^2 \, f(x)\;\;\forall \, k\in \Bbb{N},  \,\, \forall \, x \in S.
\end{equation*}
The subspace of $K(S, X)$ consisting of functions belonging to $PK_2 (S, X)$
will be denoted by $K_2(S, X)$. In other words $K_2(S, X)$ consists of
solutions of \eqref{kan} satisfying the additional condition
\begin{equation*}
f(x^k) = k \,  f(x)\;\;\forall \, k\in \Bbb{N},   \,\, \forall \, x \in S.
\end{equation*}

\begin{propos}\label{proposition3.2}
$K(S,X)=K_4(S,X)\oplus K_2(S,X)$
for any semigroup $S$ and any Banach space $X$.
\end{propos}

\begin{proof} It is clear that $K_4(S,X)\cap K_2(S,X)=\{ 0\}$.
Let $f$ be a solution of~$(\ref{kan})$. From
Lemma \ref{f(x^2)},  %Remark \ref{remark2.7},  and Remark \ref{remark2.9}
Corollary~\ref{corollary2.7} and Corollary~\ref{corollary2.9} it
follows that $f=\widehat f +\varphi$, where $\widehat f\in
K_4(S,X)$ and $\varphi \in K_2(S,X)$.
\end{proof}

\begin{propos} \label{proposition3.3}
The equation \eqref{kan} is stable for the pair $(S,X)$ if and only if
$PK_4(S,X)=K_4(S,X)$ and $PK_2(S,X)=K_2(S,X)$.
\end{propos}

\begin{proof}
Suppose that the equation \eqref{kan} is stable for the pair
$(S,X)$, and assume that $PK_4(S,X)\ne K_4(S,X)$. Let $f\in
PK_4(S,X)\setminus K_4(S,X)$. Then by
Proposition~\ref{proposition3.2} there are $\varphi_4\in K_4(S,X)$
and $\varphi_2\in K_2(S,X)$ such that for some positive $d$ we
have $|f(x)-\varphi_4(x)-\varphi_2(x)|\le d$ for all $x\in S$.
Thus the function $\psi(x)=f(x)-\varphi_4(x)-\varphi_2(x)$ is
bounded. Therefore we get $\widehat \psi =\widehat f
-\widehat\varphi_4(x)-\widehat \varphi_2(x)\equiv 0$. Now taking
into account $\widehat f = f, \, \widehat\varphi_4(x) =
\varphi_4(x), \, \widehat \varphi_2(x) \equiv 0$ we obtain
$f=\widehat\varphi_4(x)=\varphi_4(x)$. Thus we obtain a
contradiction to the assumption $f\in PK_4(S,X)\setminus
K_4(S,X)$.

\medskip
Now assume that $PK_2(S,X)\ne K_2(S,X)$. Let $f\in
PK_2(S,X)\setminus K_2(S,X)$. Then by the last proposition there
are $\varphi_4\in K_4(S,X)$ and $\varphi_2\in K_2(S,X)$ such that
for some positive $d$ we have
$|f(x)-\varphi_4(x)-\varphi_2(x)|\le d$ for all $x\in S$. The
function $\psi(x)=f(x)-\varphi_4(x)-\varphi_2(x)$ is bounded.
Therefore we obtain $\widehat \psi =\widehat f
-\widehat\varphi_4(x)-\widehat \varphi_2(x)\equiv 0$. Now taking
into account $\widehat f=0,\,
\widehat\varphi_4(x)=\varphi_4(x),\widehat \varphi_2(x)\equiv 0$
we get $0=\widehat\varphi_4(x)=\varphi_4(x)$. Hence
$|f(x)-\varphi_2(x)|\le d$.  The latter relation implies
$k\, |f(x)-\varphi_2(x)|=|f(x^k)-\varphi_2(x^k)|\le d$
for all $k \in \Bbb{N}$ and thus we see
that $f\equiv \varphi_2$. So, we obtain a contradiction to the
assumption $f\in PK_2(S,X)\setminus K_2(S,X)$.

\medskip
Therefore if equation~$(\ref{kan})$ is stable for the pair $(S,X)$,
then $PK_4(S,X)=K_4(S,X)$ and $PK_2(S,X)=K_2(S,X)$.

\medskip
Now suppose that $PK_4(S,X)=K_4(S,X)$ and $PK_2(S,X)=K_2(S,X)$.
Let us verify that equation~$(\ref{kan})$ is stable for the pair
$(S,X)$. If $f$ satisfies~(\ref{st}), then $f\in KK(S,X)$ and
there are $f_4\in PK_4(S,X)$, $f_2\in PK_2(S,X)$ and bounded
function $\delta $ such that $f=f_4+f_2+\delta$. Now from the
relations $PK_4(S,X)=K_4(S,X)$ and $PK_2(S,X)=K_2(S,X)$ we get
that $\varphi=f_4+f_2$ is a solution of~(\ref{kan}) such that
$f-\varphi$ is a bounded function. This means that
equation~(\ref{kan}) is stable for the pair $(S,X)$.
This completes the proof of the proposition.
\end{proof}

\begin{theorem}\label{rangeindependent}
Let $S$ be a semigroup, and $X$ and $E$ be two Banach spaces. Then
equation $(\ref{kan})$ is stable for the pair $(S, X)$ if and only it
is stable for the pair $(S, E)$.
\end{theorem}

\begin{proof} It is clear that we can only consider the case when
$E$ is the set of real numbers $ \mathbb{R}$.
Suppose that the equation~$(\ref{kan})$ is stable for the
pair $(S, X)$. Suppose that~$(\ref{kan})$ is not stable for the
pair $(S, \mathbb{R})$, then either $PK_4(S,\mathbb{R})\ne
K_4(S,\mathbb{R})$ or $PK_2(S,\mathbb{R})\ne K_2(S,\mathbb{R})$.
First, consider the case $PK_4(S,\mathbb{R})\ne
K_4(S,\mathbb{R})$. Let $f\in PK_4(S,\mathbb{R})\setminus
K_4(S,\mathbb{R})$.

\medskip
Let $e\in X$ and $\|e\|=1$. Consider the function $\varphi : S\to
X $ given by the formula $\varphi(x)=f(x)\cdot e$. Then from
relation
\begin{align*}
&\quad\,\, \|\; \varphi(xyz)+\varphi(x)+\varphi(y)+\varphi(z)-\varphi(xy)-\varphi(xz)-\varphi(yz)\;\| \\
&= \|\;f(xyz)\cdot e+f(x)\cdot e+f(y)\cdot e+f(z)\cdot e-f(xy)\cdot
e-f(xz)\cdot e-f(yz)\cdot e\;\| \\
&= \|\;[f(xyz)+f(x)+f(y)+f(z)-f(xy)-f(xz)-f(yz)]\cdot e\;\| \\
&=|\;f(xyz)+f(x)+f(y)+f(z)-f(xy)-f(xz)-f(yz)\;|\cdot \|e\|
\end{align*}
it follows that $\varphi\in PK_4(S,X)\setminus K_4(S,X)$ which
contradicts the fact that the equation \eqref{kan} is stable
for the pair $(S,X)$. Similarly we verify that $PK_2(S,\mathbb{R})=
K_2(S,\mathbb{R})$. So, the equation $(\ref{kan})$ is stable for the pair
$(S,\mathbb{R})$.

\medskip
Now suppose that the equation~$(\ref{kan})$ is stable for the pair
$(S,\mathbb{R})$, that is
$$PK_4(S,\mathbb{R}) = K_4(S,\mathbb{R}) \qquad {\rm and} \qquad
PK_2(S,\mathbb{R})= K_2(S,\mathbb{R}). $$
Denote by $X^*$ the
space of linear bounded functionals on $X$ endowed by functional
norm topology. It is clear that for any $\psi \in PK_i(S,X)$ and
any $\lambda \in X^*$ the function $\lambda \circ \psi $ belongs
to the space $PK_i(S,\mathbb{R})$,\; $i=2,4$ . Indeed, let for
some $c>0$ and any $x,y,z\in S$ we have
\begin{eqnarray*}
\|\; \psi(xyz)+\psi(x)+ \psi(y)+\psi(z)- \psi(xy)- \psi(xz)-
\psi(yz)\;\|\le c.
\end{eqnarray*}
Hence
\begin{align*}
\|\lambda \circ \psi(xyz)+\lambda \circ\psi(x)+\lambda \circ
\psi(y)+\lambda \circ \psi(z)-\lambda \circ \psi(xy)-\lambda \circ
\psi(xz)-\lambda \circ \psi(yz)\| \\ 
\medskip
=\|\; \lambda [\psi(xyz)+ \psi(x)+ \psi(y)+\psi(z)- \psi(xy)-
\psi(xz)- \psi(yz)] \; \|
\le c\|\lambda \|. \qquad\qquad \quad
\end{align*}

\medskip
Obviously, $\lambda \circ \psi(x^n)=n^2\lambda \circ \psi(x)$ if
$\psi\in PK_4(S,X)$ and $\lambda \circ \psi(x^n)=n\lambda \circ
\psi(x)$ if $\psi\in PK_2(S,X)$ for any $x\in S$ and for any $n
\in \mathbb{N}$.

\medskip
Hence the function $\lambda \circ \psi$ belongs to the space
$PK_4(S,\mathbb{R})\oplus PK_2(S,\mathbb{R})$. Let $f : S\to X$
belongs to the set $[PK_4(S,X)\oplus PK_2(S,X)]\setminus
[K_4(S,X)\oplus K_2(S,X)]$. Then there are $x,y,z\in S$ such that
$f(xyz)+f(x)+f(y)+f(z)-f(xy)-f(xz)-f(yz) \ne 0$. Hahn--Banach
Theorem implies that there is a $\ell \in X^*$ such that
$\ell(f(xyz)+f(x)+f(y)+f(z)-f(xy)-f(xz)-f(yz)) \ne 0 $, and we see
that $\ell \circ f$ belongs to the set
 $[PK_4(S,\mathbb{R})\oplus
PK_2(S,\mathbb{R})]\setminus [K_4(S,\mathbb{R})\oplus
K_2(S,\mathbb{R})]$.
This contradiction proves the theorem.
\end{proof}

In view of Theorem \ref{rangeindependent},
it is not important which Banach space is used on the range. Thus one may
consider the stability of the  functional equation~$(\ref{kan})$
on the pair $(S, \Bbb{R} )$.
Let us simplify the following notations: In the case $X=\mathbb{R}$ the spaces
$K(S, \mathbb{R})$, $KK(S, \mathbb{R})$, $KK_4(S, \mathbb{R})$,
$KK_2(S, \mathbb{R})$,  $PK_4(S, \mathbb{R})$, $PK_2(S, \mathbb{R})$
will be denoted by $K(S)$, $KK(S)$, $KK_4(S)$, $KK_2(S)$,  $PK_4(S)$,
$PK_2(S)$, respectively.

\begin{theorem}
In general, the functional equation~$(\ref{kan})$ is not stable on semigroups.
\end{theorem}

\begin{proof}
 Let $\mathcal F$ be a free semigroup of rank two with free
generators $a,b$.  For any word $w\in \mathcal F$.  Denote by
$\eta(w)$ the number of occurrences of\; $a^2b^2$\; in $w$. It is
easy to verify that for any $u,v\in \mathcal{F}$
\begin{equation}
\label{eta1}
 \eta(uv)-\eta(u)-\eta(v)\in  \{\;0, 1\; \}.
\end{equation}
So
\begin{eqnarray*}
\label{eta2}
 \eta(uvw)-\eta(u)-\eta(v)-\eta(w)\in  \{\;0, 1, 2\; \} ,
\end{eqnarray*}
and
\begin{eqnarray*}
|\;
\eta(xyz)+\eta(x)+\eta(y)+\eta(z)-\eta(xy)-\eta(xz)-\eta(xz)\;|\le 5.
\end{eqnarray*}
Thus we see that $\eta\in KK(\mathcal F)$, and
\begin{eqnarray*}
|\; \eta(x^2)-  2 \, \eta(x)\;|\le 1\qquad \forall \, x \in \mathcal{F}.
\end{eqnarray*}
Therefore, function
$\widetilde \eta$ defined by
\begin{eqnarray*}
\widetilde \eta(x)=\lim_{n\to \infty }\frac{1}{2^n} \, \eta(x^{2^n})
\end{eqnarray*}
belongs to $PK_2(\mathcal F)$. Let us verify that $\widetilde
\eta$ dos not belong to $K(\mathcal F)$. Indeed, it is clear that
\begin{eqnarray*}
\eta(aab^2)=1,\;
\eta(a)\;=\;\eta(b)\;=\;\eta(b^2)\;=\;\eta(a^2)\;=\;\eta(ab^2)\;=\;0,
\\
\widetilde \eta(aab^2)\;=\;1,\; \widetilde \eta(a)\;=\;\widetilde
\eta(b)\;=\;\widetilde \eta(b^2)\;=\;\widetilde
\eta(a^2)\;=\;\widetilde \eta(ab^2)\;=\;0.
\end{eqnarray*}
Therefore letting $x=a$, $y=a$, $z=b^2$, we get
\begin{eqnarray*}
\widetilde\eta(xyz)+\widetilde\eta(x)+\widetilde\eta(y)+\widetilde\eta(z)-\widetilde\eta(xy)
-\widetilde\eta(xz)-\widetilde\eta(xz)
\qquad\qquad\qquad \qquad\\
= \widetilde\eta(aab^2) + \widetilde \eta(a)+ \widetilde \eta(a)+
\widetilde \eta(b^2)- \widetilde\eta(aa)- \widetilde\eta(ab^2)-
\widetilde \eta(ab^2) =1\ne 0.
\end{eqnarray*}
So $PK_2(\mathcal F) \ne K_2(\mathcal F)$ and
equation~$(\ref{kan})$ is not stable on $\mathcal F$.
\end{proof}

\begin{definition}
An element $x$ of a semigroup $S$ is said to be {\it
periodic}  if there are $n,m\in \mathbb{N}$ such that $n\ne m$ and
$x^n=x^m$. We shall say that the semigroup is periodic if every
element of $S$ is periodic.
\end{definition}

\begin{theorem} \label{st-1}
The equation~$(\ref{kan})$ is stable for any periodic semigroup.
\end{theorem}

\begin{proof}
It is clear that if $S$ is a periodic semigroup, then
$PK_4(S)=\{0\}$ and $PK_2(S)=\{0\}$. Therefore by Theorem \ref{theorem2.10}
we have
$KK(S)=B(S)$, and
equation~$(\ref{kan})$ is stable on $S$.
\end{proof}

Now let us show that equation~$(\ref{kan})$ is stable on any
abelian semigroup $S$.
It is clear that for any abelian group $A$ and any real-valued symmetric
bimorphism $B(x,y)$ of $A \times A$,  the function $x\to B(x,x)$ belongs to
$K_4(A)$. Denote by $BM(A)$ the set of all real-valued functions $f$ on
$A$ defined by the rule $f(x)=B(x,x)$, where $B(.,.)$ is an
symmetric bimorphism.

\begin{lemma} \label{A_3-1}
Let $A_3$ be an abelian free semigroup of rank three. Then
$PK_4(A_3)=K_4(A_3)=BM(A_3)$.
\end{lemma}

\begin{proof}
Let $A_3$ be a free abelian semigroup of rank three with free
generators $a,b,c$. The space of symmetric bimorphisms on $A_3$ is
six dimensional. For $f \in K_4( A_3 )$, we choose a symmetric
bimorphism $B(x,y)$ such that $B(a,a)=f(a)$, $B(b,b)=f(b)$,
$B(c,c)=f(c)$, $B(a,b)=\frac{1}{2}[f(ab)-f(a)-f(b)]$,
$B(a,c)=\frac{1}{2}[f(ac)-f(a)-f(c)]$,
$B(b,c)=\frac{1}{2}[f(bc)-f(b)-f(c)]$.
\medskip

Hence, the function $\varphi(x)=f(x)-B(x,x)$ belongs to $PK_4(A_3)$, and
\begin{eqnarray}
\varphi(a)=\varphi(b)=\varphi(c)=\varphi(ab)=\varphi(ac)=\varphi(bc)=0.
\end{eqnarray}
We have $\varphi(a^k)=\varphi(b^k)=\varphi(c^k)=0$ for any $k\in
\Bbb{N}$. Let
\begin{eqnarray}
|\;\varphi(xyz)+\varphi(x)+\varphi(y)+\varphi(z)-\varphi(xy)-\varphi(xz)
-\varphi(yza)\;| \le \delta.
\end{eqnarray}
Then for any $p,q,k \in \Bbb{N}$ we have
\begin{equation*}
|\;\varphi(a^{pk}b^{2qk}) +
\varphi(a^{pk}) +2\varphi(b^{qk})-2\varphi(a^{pk}b^{qk})-\varphi(b^{2qk})\;|\le \delta
\end{equation*}
which simplifies to
\begin{equation*}
|\;\varphi(a^{pk}b^{2qk}) -2\varphi(a^{pk}b^{qk})\;|\le \delta .
\end{equation*}
Hence
\begin{equation*}
k^2 \, |\;\varphi(a^{p}b^{2q}) -2\varphi(a^{p}b^{q})\;|\le \delta
\end{equation*}
which is
\begin{equation*}
|\;\varphi(a^{p}b^{2q}) -2\varphi(a^{p}b^{q})\;|\le \frac{1}{k^2}\delta .
\end{equation*}
Therefore as $k \to \infty$, we obtain $\varphi(a^{p}b^{2q})=2\varphi(a^{p}b^{q})$.
Similarly, for any $p,q,k , \ell \in \Bbb{N}$ we have
\begin{equation*}
\bigg |\;\varphi(a^{pk}b^{\ell qk})+( \ell -1)[
\varphi(a^{pk})+\ell \varphi(b^{qk})]-\ell \varphi(a^{pk}b^{qk})-\frac{\ell (\ell  -1)}{2}\varphi(b^{2qk})\; \bigg |\le
\frac{\ell (\ell -1)}{2}\delta
\end{equation*}
which is
\begin{equation*}
|\;\varphi(a^{pk}b^{\ell qk})-\ell \varphi(a^{pk}b^{qk})\;|\le
\frac{\ell (\ell -1)}{2}\, \delta.
\end{equation*}
Hence
\begin{equation*}
k^2\, |\;\varphi(a^{p}b^{\ell q})-\ell \varphi(a^{p}b^{q})\;|\le
\frac{\ell (\ell -1)}{2}\, \delta
\end{equation*}
which is
\begin{equation*}
|\;\varphi(a^{p}b^{\ell q})-\ell \varphi(a^{p}b^{q})\;|\le
\frac{\ell (\ell -1)}{2k^2}\, \delta.
\end{equation*}
Therefore as $k \to \infty$, we get $\varphi(a^{p}b^{\ell q})=\ell
\varphi(a^{p}b^{q})$. Similarly we obtain $\varphi(a^{\ell
p}b^{q})=\ell \varphi(a^{p}b^{q})$. So, for any $n,m\in \Bbb{N}$,
we get $\varphi(a^n b^m)=n m\, \varphi(ab) = 0$.

\medskip
The same way we obtain equalities $\varphi(a^nc^m) = nm\,
\varphi(ac)=0$ and $\varphi(b^nc^m) = nm\, \varphi(bc)=0$.

\medskip
Now for any $n,m,k, \ell \in \Bbb{N}$ we have
\begin{equation*}
|\;\varphi(a^{pk}b^{qk}c^{\ell k})- \varphi(a^{pk})-
\varphi(b^{qk}) - \varphi(c^{\ell k}) - \varphi(a^{pk}b^{qk})
 - \varphi(a^{pk}c^{\ell k}) - \varphi(b^{qk}c^{\ell k})
\;|\le \delta .
\end{equation*}
Hence $ |\;\varphi(a^{pk}b^{qk}c^{\ell k}) \;|\le \delta $, and we have
$k^2 \, |\;\varphi(a^{p}b^{q}c^{{\ell}}) \;|\le \delta$. Thus
$$|\;\varphi(a^{p}b^{q}c^{\ell}) \;|\le\frac{1}{k^2} \, \delta . $$
By taking the limit
as $k \to \infty$, we see that $\varphi(a^{p}b^{q}c^{\ell}) = 0$.
It means that
$$f(x)=B(x,x)\in BM(A_3)$$
and the proof of the lemma is now finished.
\end{proof}

For any group $G$, we will denote by $X(G)$ the set of real-valued additive
characters of $G$.

\begin{lemma}
Let $A_3$ be an abelian free semigroup of rank three. Then
$PK_2(A_3) = K_2(A_3)=X(A_3)$.
\end{lemma}

\begin{proof}
Let $f\in PK_2(A_3)$ and $f(a)=p,\, f(b)=q,\, f(c)=r$. Further,
let $\psi $ be an additive character of $A_3$ such that
$\psi(a)=p,\;\; \psi(b)=q,\;\; \psi(c)=r$. Then the function
$\varphi(x)=f(x)-\psi(x)$ belongs to $PK_2(A_3)$ and satisfies the
condition $\varphi(a)= \varphi(b)= \varphi(c)=0$. Let us show that
$\varphi\equiv 0$.

\medskip
Let $\delta $ be a positive number such that for any $x,y,z\in A_3$
\begin{equation*}
 |\;\varphi(xyz)+ \varphi(x)+ \varphi(y)+
\varphi(z)-\varphi(xy)-\varphi(xz)-\varphi(yz)\;| \le \delta.
\end{equation*}

Then for any $p,q,k , \ell \in \Bbb{N}$ we have
\begin{equation*}
\bigg |\;\varphi(a^{pk}b^{\ell qk})+( \ell -1)[
\varphi(a^{pk})+\ell \varphi(b^{qk})]-\ell
\varphi(a^{pk}b^{qk})-\frac{\ell (\ell  -1)}{2}\varphi(b^{2qk})\;
\bigg |\le \frac{\ell (\ell -1)}{2}\delta
\end{equation*}
which is
\begin{equation*}
|\;\varphi(a^{pk}b^{\ell qk})-\ell \varphi(a^{pk}b^{qk})\;|\le
\frac{\ell (\ell -1)}{2}\, \delta.
\end{equation*}
Hence
\begin{equation*}
k\, |\;\varphi(a^{p}b^{\ell q})-\ell \varphi(a^{p}b^{q})\;|\le
\frac{\ell (\ell -1)}{2}\, \delta
\end{equation*}
which is
\begin{equation*}
|\;\varphi(a^{p}b^{\ell q})-\ell \varphi(a^{p}b^{q})\;|\le
\frac{\ell (\ell -1)}{2k}\, \delta.
\end{equation*}
Therefore as $k \to \infty$, we get $\varphi(a^{p}b^{\ell q})=\ell
\varphi(a^{p}b^{q})$. Similarly we obtain $\varphi(a^{\ell
p}b^{q})=\ell \varphi(a^{p}b^{q})$. So, for any $n,m\in \Bbb{N}$,
we get $\varphi(a^n b^m)=n m\, \varphi(ab)$. It follows that for
$u=a^nb^m$ we have
$\varphi(u^k)=\varphi(a^{kn}b^{km})=k^2nm\varphi(ab)=k^2\varphi(u)$
But $\varphi(x)\in PK_2(A_3)$, therefore we have
$\varphi(u^k)=k\varphi(u)=k^2\varphi(u)$. The last relation
implies $\varphi(u)=0$.

\medskip
The same way we obtain equalities $\varphi(a^nc^m) =0$ and
$\varphi(b^nc^m) = 0$ for any $n,m\in \Bbb{N}$.

\medskip
Now for any $n,m,k, \ell \in \Bbb{N}$ we have
\begin{equation*}
|\;\varphi(a^{pk}b^{qk}c^{\ell k})- \varphi(a^{pk})-
\varphi(b^{qk}) - \varphi(c^{\ell k}) - \varphi(a^{pk}b^{qk})
 - \varphi(a^{pk}c^{\ell k}) - \varphi(b^{qk}c^{\ell k})
\;|\le \delta .
\end{equation*}
Hence $ |\;\varphi(a^{pk}b^{qk}c^{\ell k}) \;|\le \delta $, and we
have $k \, |\;\varphi(a^{p}b^{q}c^{{\ell}}) \;|\le \delta$. Thus
$$|\;\varphi(a^{p}b^{q}c^{\ell}) \;|\le\frac{1}{k^2} \, \delta . $$
By taking the limit as $k \to \infty$, we see that
$\varphi(a^{p}b^{q}c^{\ell}) = 0$.

\medskip
Therefore, $\varphi\equiv 0$ and $f\equiv \psi\in X(A_2)$.
\end{proof}

\begin{theorem} \label{abelian} 
Let $A$ be any abelian group. Then
$PK(A)=K_4(A)\oplus K_2(A)$.
\end{theorem}
\begin{proof}
Let show that $PK_4(A)= K_4(A)$ and  $PK_2(A)= K_2(A)$.
\medskip
Suppose that $PK_4(A)\ne K_4(A)$. In this case there are $f\in
PK_4(A)$ and $x,y,z\in A$ such that
\begin{eqnarray*}
|\; f(xyz)+f(x)+f(y)+f(z)-f(xy)-f(xz)-f(yz)\;|=d>0.
\end{eqnarray*}
Denote by $B$ the subsemigroup of $A$ generated by three elements
$x,y,z$. Let $\tau $ be an epimorphism of $A_3$ onto $B$ given by
the rule $\tau(a)=x,\;\tau(b)=y,\; \tau(c)=z$. So, if we consider
function $g(t)=f(\tau(t))$ we get an element of $PK_4(A_3)$ such
that
\begin{eqnarray*}
|\; g(abc)+g(a)+g(b)+g(c)-g(ab)-g(ac)-g(bc)\;|=d>0
\end{eqnarray*}
which contradicts Lemma~\ref{A_3-1}.

\medskip
Similarly, we come to a contradiction if we suppose that
$PK_2(A)\ne K_2(A)$. Hence $PK_4(A)= K_4(A)$, and
$PK_2(A)=K_2(A)$.
\end{proof}

\begin{corollary}
\label{jung-th=} Suppose $A$ is an abelian group. Then 
$$K(A)=PK(A)=K_4(A)\oplus K_2(A).$$
\end{corollary}

Now from Proposition~\ref{proposition3.3} we get the following
corollary.

\begin{corollary}
%\label{abelian} 
The equation~$(\ref{kan})$ is stable on any
abelian semigroup $A$.
\end{corollary}

For any group $G$, let $Q(G)$ be the set of solutions of the
quadratic functional equation
\begin{eqnarray*}
f(xy)+f(xy^{-1})=2f(x)+2f(y).
\end{eqnarray*}
Moreover, we denote by $PK^+(G)$ and by  $PK^-(G)$  subspaces of  $PK(G)$
consisting of functions $f$ such that $f(x^{-1})=f(x)$ and
$f(x^{-1})=-f(x)$, respectively.

\begin{lemma} %\label{} 
For any group $G$ the following relations
$$PK_4(G)=PK^+(G),\;\; PK_2(G)=PK^-(G)$$ hold.
\end{lemma}

\begin{proof} 
It is clear that $PK_4(G)\subseteq PK^+(G)$, and $PK_2(G)\subseteq
PK^-(G)$. Let us show that $PK^+(G)\subseteq PK_4(G)$, and
$PK^-(G)\subseteq PK_2(G)$, respectively.
Let $f\in PK(G)$, then there are
$\varphi\in PK_4(G)$ and $\psi\in PK_2(G)$ such that
$f(x)=\varphi(x)+\psi(x)$. If $f\in PK^+(G)$, then
\begin{eqnarray*}
f(x)=f(x^{-1})=\varphi(x^{-1})+\psi(x^{-1})=\varphi(x)-\psi(x),
\end{eqnarray*}
so $\varphi(x)+\psi(x)=\varphi(x)-\psi(x)$ and we see that
$\psi(x)\equiv 0$ and $f(x)=\varphi(x)\in PK_4(G)$. Therefore
$PK^+\subseteq PK_4(G)$.
\medskip

Now if  $f\in PK^-(G)$, then
\begin{eqnarray*}
f(x)=-f(x^{-1})=-\varphi(x^{-1})-\psi(x^{-1})=-\varphi(x)+\psi(x),
\end{eqnarray*}
so $\varphi(x)+\psi(x)=-\varphi(x)+\psi(x)$ and we see that
$\varphi\equiv 0$ and $f(x)=\psi(x)\in PK_2(G)$. Therefore
$PK^-(G)\subseteq PK_2(G)$.
\end{proof}

\begin{lemma}
\label{quadratic} Let $G$ be an arbitrary group and $f\in Q(G)$,
then for any $x,y,z\in G$ we have
\begin{equation}
\label{quadr} 
f(xyz)+f(xzy)=f(x)+3f(y)+3f(z)+f(xy) +f(xz)-f(yz).
\end{equation}
\end{lemma}

\begin{proof}
We have
\begin{align*}
f(xyz)+f(xyz^{-1}) &= 2\, f(xy)+2 \, f(z), \\
f(xyz^{-1})+f(xzy^{-1}) &= 2\, f(x) + 2\, f(yz^{-1}), \\
f(xzy^{-1})+f(xzy) &= 2\, f(xz) + 2\, f(y).
\end{align*}
Therefore
\begin{align*}
f(xyz)+f(xzy) = f(xy)+f(z) +f(x)+f(yz^{-1}) +f(xz)+f(y).
\end{align*}
Now taking into account relations
\begin{eqnarray*}
f(yz)+f(yz^{-1})=2f(y)+2f(z),
\\
f(yz^{-1})=2f(y)+2f(z)-f(yz)
\end{eqnarray*}
we get
\begin{align*}
f(xyz)+f(xzy) &= f(xy)+f(z) +f(x)+2f(y)+2f(z)-f(yz) +f(xz)+f(y) \\
& = f(xy) +f(x)+3f(y)+3f(z)-f(yz) +f(xz).
\end{align*}
This completes the proof of the lemma.
\end{proof}

\begin{lemma} \label{qA_3-1}
Let $\overline A_3$ be an abelian free group of rank three. Then
$Q(\overline A_3)=BM(\overline A_3)$.
\end{lemma}

\begin{proof}
Let $\overline A_3$ be a free abelian group of rank three with
free generators $a,b,c$. It is clear that $BM(\overline A_3 )\subseteq Q(\overline A_3 )$.
The space of symmetric bimorphisms on $\overline A_3$ is six dimensional. For
$f \in Q( \overline A_3 )$, we choose a symmetric bimorphism $B(x,y)$ such
that $B(a,a)=f(a)$, $B(b,b)=f(b)$, $B(c,c)=f(c)$,
$B(a,b)=\frac{1}{2}[f(ab)-f(a)-f(b)]$,
$B(a,c)=\frac{1}{2}[f(ac)-f(a)-f(c)]$,
$B(b,c)=\frac{1}{2}[f(bc)-f(b)-f(c)]$.
\medskip

Hence, the function $\varphi(x)=f(x)-B(x,x)$ belongs to $Q( \overline A_3 )$,
and
\begin{eqnarray}
\varphi(a)=\varphi(b)=\varphi(c)=\varphi(ab)=\varphi(ac)=\varphi(bc)=0.
\end{eqnarray}
We have $\varphi(a^k)=\varphi(b^k)=\varphi(c^k)=0$ for any $k\in \Bbb{Z}$.

\medskip
Now from~\eqref{quadr} we get
\begin{equation}
\label{quadr1}
2\varphi(xyz)=\varphi(x)+3\varphi(y)+3\varphi(z)+\varphi(xy)
+\varphi(xz)-\varphi(yz).
\end{equation}
From this equality we get
\begin{align*}
%\label{quadr12}
2\varphi(a^nb^mb^k) &=\varphi(a^n)+3\varphi(b^m)+3\varphi(b^k)+\varphi(a^nb^m)
+\varphi(a^nb^k)-\varphi(b^mb^k), \\
2\varphi(a^nb^mb^k) &=\varphi(a^nb^m) +\varphi(a^nb^k).
\end{align*}
So, if $k=0$, then $2\varphi(a^nb^m)=\varphi(a^nb^m)$ and we see
that $\varphi(a^nb^m)=0$ for any $n,m\in \Bbb{Z}$. Similarly we verify
that $\varphi(a^nc^m)=0,\,\varphi(b^nc^m)=0$. Now
from~\eqref{quadr1} we get
\begin{align*}
\label{quadr13}
2\varphi(a^nb^mc^k) &=\varphi(a^n)+3\varphi(b^m)+3\varphi(c^k)+\varphi(a^nb^m)
+\varphi(a^nc^k)-\varphi(b^mc^k) \\
&= 0.
\end{align*}
It means that
$$f(x)=B(x,x)\in BM( \overline A_3 )$$
and the proof of the lemma is now finished.
\end{proof}

\begin{lemma} %\label{} 
Let $\overline A_3$ be an abelian group of rank three,
then $K_2(\overline A_3)=X(\overline A_3)$.
\end{lemma}
\begin{proof} The proof is similar to the proof of
Lemma~\ref{qA_3-1}.
\end{proof}

\begin{propos}\label{prop3.17}
Let $A$ be an abelian group, then
\begin{eqnarray*}
PK(A)=PK_4(A)\oplus PK_2(A)=K_4(A)\oplus K_2(A)=Q(A)\oplus X(A).
\end{eqnarray*}
Another words general solution $f: A\to R$ of equation~\eqref{kan}
is of the form
$$
f(x)=B(x,x)+\psi(x),
$$
where $B(x,y)\in BM(A),\; \psi\in X(A)$.
\end{propos}
\begin{proof}
First we verify equality $K_4(A)=Q(A)$. We have $K_4(A)\subseteq
Q(A)$. Suppose that there is $f \in Q(A) \setminus K_4(A)$, then
there are $x,y,z\in A$ such that
\begin{eqnarray}
\label{ne1} f(xyz)+f(x)+f(y)+f(z)\ne f(xy)+f(xz)+f(yz) .
\end{eqnarray}
Now let $\overline A_3$ be a free abelian group with free
generators $a,b,c$. Let $B$ be a subgroup of $A$ generated by
elements $x,y,z$ and let $\pi: \overline A_3 \to B$ be an
epimorphism such that $\pi(a)=x,\, \pi(b)=y, \, \pi(c)=z$. Then function
$\varpi(t) = f(\pi(t))$ is an element of $Q(\overline A_3)$. By
Lemma~\ref{qA_3-1} we have $\varpi\in BM(\overline A_3)$. But this
contradicts to~\eqref{ne1} because
\begin{align*}
\varpi(abc)&+\varpi(a)+\varpi(b)+\varpi(c)-\varpi(ab)-\varpi(ac)-\varpi(bc) \\
&= f(xyz)+f(x)+f(y)+f(z)- f(xy)-f(xz)-f(yz) \ne 0.
\end{align*}
Therefore, $f\in K_4(A)$. Similar way we verify that
$K_2(A)=X(A)$.
\end{proof}

As a first corollary of Proposition~\ref{prop3.17} we obtain the following
corollary that generalizes Kannappan's result \cite{ka} (see Introduction) in the  
case $\Bbb{K} = \Bbb{R}$.

\begin{corollary} %\label{} 
Let $A$ be an abelian group, then general solution $f:
A\to \Bbb{R}$ of equation~\eqref{kan} is of the form
$$
f(x)=B(x,x)+\psi(x),
$$
where $B(x,y)$ is an symmetric bimorphism and $\psi\in X(A)$.
\end{corollary}

From Proposition~\ref{prop3.17} we obtain the following
two theorems that generalize the results of Jung \cite{jung1} mentioned in the  
Introduction.

\begin{theorem}
Suppose that $A$ is an abelian group, and $X$ a real Banach space.
Let $f: A\to X$ satisfies the inequalities
\begin{eqnarray*}
\|\,f(xyz)+f(x)+f(y)+f(z)-f(xy)-f(xz)-f(yz)\,\|\le d  \\
\|\,f(x)-f(-x)\,\|\le \theta
\end{eqnarray*}
for some $d, \theta>0$ and for all $x,y,z\in A$. Then there exists
a unique quadratic mapping $q: A\to X$ which satisfies
\begin{eqnarray*}
\|\,f(x)-q(x)\,\|\le \delta
\end{eqnarray*}
for some positive $\delta$ and all $x\in A$.
\end{theorem}

\begin{proof}
According to the Theorem~\ref{rangeindependent} we can assume that
$X=\Bbb{R}$. From Theorem~\ref{theorem2.10} and Proposition~\ref{prop3.17}
it follows that there are $q(x)\in Q(A)$, $\psi\in X(A)$  and
$\gamma\in B(A)$ such that $f(x)=q(x)+\psi(x)+\gamma(x)$.
Therefore,
\begin{align*}
|\,f(x)-f(x^{-1}) \,|&= |\,
q(x)+\psi(x)+\gamma(x)-q(x^{-1})-\psi(x^{-1})-\gamma(x^{-1}) \,| \\
&= |\, 2\psi(x)+\gamma(x)-\gamma(x^{-1}) \, |\le \theta,
\end{align*}
and we see that $\psi(x) $ is a bounded function. Hence,
$\psi\equiv 0$ and $f(x) = q(x)+\gamma(x)$. If $\delta$ is a
positive real number such that $|\gamma(x)|\le \delta$\, for all $x\in
A$, then we have $|f(x)-q(x)|\le \delta$\, for all $x\in A$. The
proof is now complete.
\end{proof}

\begin{theorem}
Suppose that $A$ is an abelian group, and $X$ a real Banach space.
Let $f: A\to X$ satisfies the inequalities
\begin{eqnarray*}
\|\,f(xyz)+f(x)+f(y)+f(z)-f(xy)-f(xz)-f(yz)\,\|\le d
\\
\|\,f(x)+f(-x)\,\|\le \theta
\end{eqnarray*}
for some $d, \theta>0$ and for all $x,y,z\in A$. Then there exists
a unique additive mapping $\psi: A\to X$ which satisfies
\begin{eqnarray*}
\|\,f(x)-\psi(x)\,\|\le \delta
\end{eqnarray*}
for some positive $\delta$ and all $x\in A$.
\end{theorem}
\begin{proof} The proof is similar to that of the previous
theorem.
\end{proof}

%%%%%%%%%%%%%%%%%%%%%%%%%%%%%%%%%%%%%%%%%%%%%%%%%%%%%%%%%
%%
%%                    Section 4: Embedding
%%
%%%%%%%%%%%%%%%%%%%%%%%%%%%%%%%%%%%%%%%%%%%%%%%%%%%%%%%%%

\section{Embedding}

\begin{remark} \label{rem1}
If $S$ is a semigroup with zero and $f\in KK(S)$, then $f$ is bounded.
\end{remark}

\begin{proof}
Since $f\in KK(S)$, the function $f$ satisfies
\begin{eqnarray*}
|\;f(xyz)+f(x)+f(y)+f(z)-f(xy)-f(xz)-f(yz)\;|\le d,
\end{eqnarray*}
for all $x,y,z\in S$ and for some $d > 0$.
If we put $y=z=0$ in the above inequality, we obtain
\begin{equation*}
|\;f(0)+f(x)+f(0)+f(0)-f(0)-f(0)-f(0)\;|\le d.
\end{equation*}
Therefore $|\;f(x)\;|\le d $.
So $f$ is a bounded function.
\end{proof}

The following corollary follows from the Remark \ref{rem1}.
\begin{corollary} %%{Corollary 5}
\label{cor5}
Let $S_0$ be a semigroup obtained by adjoining the zero to the arbitrary
semigroup $S$. Then $S$ can be embedded into the semigroup
$S_0$ such that the equation $(\ref{kan})$ is stable on $S_0$.
\end{corollary}

\begin{proof}
From Remark \ref{rem1}, we have $KK(S_0)= B(S_0)$.
Hence the equation $(\ref{kan})$ is stable on $S_0$.
\end{proof}

\begin{definition}
We shall say that in a semigroup $S$ a {\it left law of reduction}
is fulfilled if any equality  $xy =xz$ in $S$ implies $y = z$.
Similarly, we shall say that in a semigroup $S$ a {\it right law
of reduction} is fulfilled if any equality  $yx =zx$ in $S$
implies $y =z$.
\end{definition}

Obviously in a semigroup with zero neither left nor right law of
reduction is fulfilled.

\medskip
The embedding presented in Corollary \ref{cor5}
does not preserve some important properties of
semigroup. For instance, if $S$ is a group $S_0$ is not necessarily a
group. Similarly, if $S$ is a semigroup with law of reduction, then $S_0$
does not have the same property.

\medskip
Our main goal in this section is to construct another embedding 
preserving properties of semigroups such as laws of reduction
and the axioms of a group. From now on let $S$ be an arbitrary
semigroup with unit $e$.

\begin{lemma} 
\label{qkan-1}
Let $f\in KK(S)$ so that
\begin{equation}\label{qkan2}
|\;   f(xyz)+f(x)+f(y)+f(z)-f(xy)-f(xz)-f(yz) \;|\le d
\end{equation}
for any $x, y, z\in S$ and for some $d > 0$.
Further, let $c$ be an element of order two. Then
\begin{align}
\label{cu}
 |\;  f(u) - f(cu)  \;|&\le 2 \, d , \\
\label{uc}
 |\; f(u) - f(uc)  \;|&\le 2 \, d, \\
\label{u^c}
 |\; f(u^c)-f(u) \;|&\le 8 \, d
\end{align}
for any $u\in S$.
\end{lemma}

\begin{proof}
Letting $x=y=z=e$ in $\eqref{qkan2}$, we obtain $|f(e) | \leq d$. Similarly, letting
$x=y=z=c$ in $(\ref{qkan2})$, we have $|\;   f(ccc)+3f(c)-3f(cc) \;| \le d$. Since
$c$ is an element of order two, the last inequality reduces to $|\; 4f(c)-3f(e) \;|\le d$.
Hence we have $|\; 4f(c) \;|\le d +3|f(e)|$ and consequently
\begin{equation*}\label{c}
 |\;f(c)\;|\le \frac{1}{4}\, d+\frac{3}{4}\, f(e)\le d.
\end{equation*}

\medskip
Next substituting $x=c$, $y=c$ and $z=u$ in $\eqref{qkan2}$, we have
\begin{equation*}
|\;   f(ccu)+f(c)+f(c)+f(u)-f(cc)-f(cu)-f(cu) \;|\le d .
\end{equation*}
Since $c$ is an element of order two, the last inequality yields
$$ |\;   2f(u)+2f(c)-f(e)-2f(cu)\;|\le d$$
and hence we have
$ |\;   2f(u)-2f(cu)\;|\le d+3\, d$. Therefore simplifying, we see that
\begin{equation*}
\label{cu2} |\;   f(u)-f(cu)\;|\le 2\, d
\end{equation*}
which is \eqref{cu}.

\medskip
Similarly, letting $x=u$, $y=c$ and $z=c$ in $\eqref{qkan2}$, we get
\begin{equation*}
|\;   f(ucc)+f(c)+f(c)+f(u)-f(uc)-f(uc)-f(cc) \;|\le d .
\end{equation*}
Using the fact that $c$ is of order two, we have
$$ |\;   2f(u)+2f(c)-2f(uc)-f(e)\;|\le d . $$
This last inequality yields
$ |\;   2f(u)-2f(uc)\;|\le d+3\, d$. Simplifying, we get
$$|\;   f(u)-f(uc)\;|\le 2\, d$$
which is \eqref{uc}.

\medskip
Again, substituting $x=c$, $y=u$ and $z=c$ in $\eqref{qkan2}$, we obtain
\begin{equation*}
|\;   f(cuc)+f(c)+f(c)+f(u)-f(cu)-f(uc)-f(cc) \;|\le d .
\end{equation*}
Using the fact that $c$ is of order two and simplifying, we have
$$ |\;   f(u^c)+f(u) -f(cu) -f(uc)\;|\le d+3\, d .$$
Using \eqref{cu} we obtain
\begin{align*}
|\;   f(u^c)-f(uc)\;| &= | \;  f(u^c) + f(u) - f(cu) - f(uc) + f(cu) - f(u) \;|  \\
&\leq  | \;  f(u^c) + f(u) - f(cu) - f(uc) \;|  +  | \; f(cu) - f(u) \;| \\
&\leq 4\, d + 2\, d = 6\, d .
\end{align*}
Similarly using \eqref{uc} we have
$$ |\;   f(u^c)-f(cu)\;|\le 6\, d. $$
Now taking into account the last inequality and \eqref{cu}, we get
\begin{align*}
|\;   f(u^c)-f(u)\;| &= |\;  f(u^c)- f(cu) + f(cu) - f(u) \; | \\
&=  |\;  f(u^c)- f(cu) \; | + |\; f(cu) - f(u) \; | \\
&\leq 6\, d + 2\, d =  8\, d.
\end{align*}
The proof of the lemma is now complete.
\end{proof}

Now consider semidirect product $H=K \semidirect S$ of semigroup $S$ and
a group $K$, where elements of $K$ act on $S$ by automorphisms.
Also we suppose that every non unit element of $K$ has order two.

\begin{lemma}\label{lemma4.7}
Suppose that  $f\in PK_4(G)$ and satisfies  condition
~$(\ref{qkan2})$ on $H$. Let $b, c, bc \in K$ be the elements of
order two. Suppose for $u\in S$ the elements $u^{bc}, \, u^c, \,
u$ generate an abelian subsemigroup, then
\begin{equation}\label{9k}
f(u^{bc}u^cu)=9\, f(u) \qquad \forall \, u\in S.
\end{equation}
\end{lemma}

\begin{proof}
Using Lemma \ref{lemma2.2} with $n=5$ and
$x_1=u, x_2=b, x_3=u, x_4=c, x_5=u$, we get
\begin{align*}
|\;f(ubucu)+3[ 3f(u) &+ f(b) + f(c)] - 3f(u^2) - f(ub) \qquad\qquad\qquad \\
 &- \, 2 f(uc) - 2 f(bu) - f(cu) - f(bc) \;|\le 6d.
\end{align*}
Now taking into account relations $|f(b)|\le d, \, |f(c)|\le d$, and
$|f(bc)|\le d$, we obtain

\begin{equation*}
|\;f(ubucu)+9f(u)-3f(u^2)-f(ub)-2f(uc)-2f(bu)-f(cu) \;|\le 13\, d
\end{equation*}
which is
\begin{equation*}
|\;f(bcu^{bc}u^cu)+9f(u)-3f(u^2)-f(ub)-2f(uc)-2f(bu)-f(cu) \;|\le
13 \, d.
\end{equation*}
Now using~$(\ref{uc})$ and ~$(\ref{cu})$
 we get
\begin{equation*}
|\;f(u^{bc}u^cu)+9f(u)-3f(u^2)-f(u)-2f(u)-2f(u)-f(u) \;|\le 13 \,
d+14 \, d=27\,d.
\end{equation*}
Using $ f( u^2) = 4\, f(u)$, the last inequality yields
\begin{equation*}
|\;f(u^{bc}u^cu)-9f(u)\;|\le 27\, d.
\end{equation*}
Therefore for any $n \in \Bbb{N}$ we have
\begin{align*}
n^2|f(u^{bc}u^cu)-9f(u)|&=|f((u^{bc}u^cu)^n)-9f(u^n)|\\
&=|f((u^n)^{bc}(u^n)^cu^n)-9f(u^n)|\le 27 d.
\end{align*}
Thus we have
\begin{equation*}
\;f(u^{bc}u^cu)=9 \, f(u)
\end{equation*}
and the proof of the lemma is complete.
\end{proof}

\begin{lemma}\label{lemma4.9}
Let $f$ be an element of  $PK_2(G)$ satisfying condition
~$(\ref{qkan2})$ on $H$.  Let $b, c, bc \in K$ be the elements of
order two. Suppose for $u\in S$ the elements $u^{bc}, \, u^c, \,
u$ generate an abelian subsemigroup, then
\begin{equation}\label{3k}
f(u^{bc}u^cu)=3\, f(u) \qquad \forall \, u\in S.
\end{equation}
\end{lemma}

\begin{proof}
Using Lemma \ref{lemma2.2} with $n=5$ and $x_1=u, x_2=b, x_3=u,
x_4=c, x_5=u$, we get
\begin{align*}
|\;f(ubucu)+3[ 3f(u) &+ f(b) + f(c)] - 3f(u^2) - f(ub) \qquad\qquad\qquad \\
 &- \, 2 f(uc) - 2 f(bu) - f(cu) - f(bc) \;|\le 6d.
\end{align*}
Now taking into account relations $|f(b)|\le d, \, |f(c)|\le d$,
and $|f(bc)|\le d$, we obtain

\begin{equation*}
|\;f(ubucu)+9f(u)-3f(u^2)-f(ub)-2f(uc)-2f(bu)-f(cu) \;|\le 13\, d
\end{equation*}
which is
\begin{equation*}
|\;f(bcu^{bc}u^cu)+9f(u)-3f(u^2)-f(ub)-2f(uc)-2f(bu)-f(cu) \;|\le
13 \, d.
\end{equation*}
Now using~$(\ref{uc})$ and ~$(\ref{cu})$
 we get
\begin{equation*}
|\;f(u^{bc}u^cu)+9f(u)-3f(u^2)-f(u)-2f(u)-2f(u)-f(u) \;|\le 13 \,
d+14 \, d=27\,d.
\end{equation*}
Using $ f( u^2) = 2 \, f(u)$, the last inequality yields
\begin{equation*}
|\;f(u^{bc}u^cu)-3f(u)\;|\le 27\, d.
\end{equation*}
Therefore for any $n \in \Bbb{N}$ we have
\begin{align*}
n^2|f(u^{bc}u^cu)-3f(u)|&=|f((u^{bc}u^cu)^n)-3f(u^n)|\\
&=|f((u^n)^{bc}(u^n)^cu^n)-3f(u^n)|\le 27 d.
\end{align*}
Thus we have
\begin{equation*}
\;f(u^{bc}u^cu)=3 \, f(u)
\end{equation*}
and the proof of the lemma is complete.
\end{proof}

Let $S$ be an arbitrary semigroup with unit and $B$ a group. For
each $b\in B$ denote by $S(b)$ a group that is isomorphic to $S$
under isomorphism $a\to a(b)$. Denote by $H= S^{(B)}=\prod_{b\in
B}S(b)$ the direct product of groups $S(b)$. It is clear that if
$a_1(b_1)a_2(b_2)\cdots a_k(b_k)$ is an element of $H$, then for
any $b\in B$, the mapping
$$
b^* : a_1(b_1)a_2(b_2)\cdots a_k(b_k)\to a_1(b_1b)a_2(b_2b)\cdots a_k(b_kb)
$$
is an automorphism of $D$ and $b\to b^*$ is an embedding of $B$
into $Aut\,H$. Thus, we can form a semidirect product $G= B\semidirect
H$. This semigroup is called {\it the wreath product\/} of the
semigroup $S$ and the group $B$, and will be denoted by $G=S\wr B$. We
will identify the group $S$ with subgroup $S(1)$ of $H$, where
$1\in B$. Hence, we can assume that $S$ is a subgroup of $H$.

\medskip
Let us denote, by $C$, the group of order four having generators
$b, c$ and defining relations: $b^2=c^2=1, bc=cb$.
Consider the semigroup group $S\wr C$.

\begin{lemma}
\label{lemma-2} 
Suppose that $f\in PK_4(S\wr C)$. If for some
$x,y,z\in S$ we have
$$
|f(xyz)+f(x)+f(y)+f(z)-f(xy)-f(xz)-f(yz)|=\delta > 0
$$
then for some $x_1,y_1,z_1\in H$ we have
$$
|f(x_1y_1z_1)+f(x_1)+f(y_1)+f(z_1)-f(x_1y_1)-f(x_1z_1)-f(y_1z_1)|=9\delta.
$$
\end{lemma}
\begin{proof}
Let $x_1=xyz$, $y_1=x_1^b$, $z_1=x_1^c$. We have $x_1\in S(1),
x_1^b\in S(b), x_1^c\in S(c)$, therefore subsemigroup generated by
$x_1, x_1^b, x_1^c$ is an abelian semigroup. Applying
Lemma~\ref{lemma4.7} we get
\begin{equation*}
f(xyz(xyz)^b(xyz)^c)=f(xx^bx^cyy^by^czz^bz^c),
\end{equation*}
\begin{align*}
&|\,f(xx^bx^cyy^by^czz^bz^c)+f(xx^bx^c) +f(yy^by^c)+f(zz^bz^c)  \\
&\qquad\qquad\quad - \, f(xx^bx^cyy^by^c)-f(xx^bx^czz^bz^c)-f(yy^by^czz^bz^c)\,| \\
&\qquad =9|f(xyz)+f(x)+f(y)+f(z)-f(xy)-f(xz)-f(yz)|=9\delta
\end{align*}
and the proof is now complete.
\end{proof}

\begin{lemma} \label{lemma-2a} 
Suppose that $f\in PK_2(S\wr C)$. If for some
$x,y,z\in S$ we have
$$
|f(xyz)+f(x)+f(y)+f(z)-f(xy)-f(xz)-f(yz)|=\delta > 0
$$
then for some $x_1,y_1,z_1\in H$ we have
$$
|f(x_1y_1z_1)+f(x_1)+f(y_1)+f(z_1)-f(x_1y_1)-f(x_1z_1)-f(y_1z_1)|=3\delta.
$$
\end{lemma}
\begin{proof}
Let $x_1=xyz$, $y_1=x_1^b$, $z_1=x_1^c$. We have $x_1\in S(1), \, 
x_1^b\in S(b), \, x_1^c\in S(c)$, therefore subsemigroup generated by
$x_1, \, x_1^b, \, x_1^c$ is an abelian semigroup. Applying
Lemma~\ref{lemma4.9} we get
\begin{equation*}
f(xyz(xyz)^b(xyz)^c)=f(xx^bx^cyy^by^czz^bz^c),
\end{equation*}
\begin{align*}
&|\,f(xx^bx^cyy^by^czz^bz^c)+f(xx^bx^c) +f(yy^by^c)+f(zz^bz^c) \\
&\qquad\qquad\quad - \, f(xx^bx^cyy^by^c)-f(xx^bx^czz^bz^c)-f(yy^by^czz^bz^c)\,| \\
&\qquad =3\, |f(xyz)+f(x)+f(y)+f(z)-f(xy)-f(xz)-f(yz)|=3\delta
\end{align*}
and the proof is now complete.
\end{proof}

\begin{theorem} Let $S$ be a semigroup
with left (or right) law of reduction. Then $S$ can be embedded
into a semigroup $G$ with the left (or right respectively) law of
reduction and the equation $(\ref{kan})$ is stable on $G$.
Moreover,  if $S$ is a group then $G$ is a group too.
\end{theorem}

\begin{proof}
Let $C_i$, for $i\in \N$, be a group of order 4 with two
generators $b_i, c_i$ and defining relations $b_i^2=1, c_i^2=1,
b_ic_i=c_ib_i$. Consider the chain of groups defined as follows:
$$
S_1=S,\, S_2=S_1\wr C_1,\, S_3=S_2\wr C_2,\, \dots ,
\,S_{k+1}=S_k\wr C_k, \dots
$$
Define a chain of embeddings
\begin{equation}
\label{chain} S_1=S\,\to \,  S_2=S_1\wr C_1\,\to \, S_3=S_2\wr
C_2\,\to \, \dots \to \, \,S_{k+1}=S_k\wr C_k\to \, \dots
\end{equation}
by identifying $S_k$ with $S_k(1)$ a subgroup of $S_{k+1}$. Let
$G$ be the direct limit of the chain~(\ref{chain}). Then we have
$G=\cup_{k\in \N}S_k$ and
$$
S_1\subset S_2\subset \dots \subset S_k \subset S_{k+1} \subset
\dots \dots\,\, \subset G.
$$
Let $f\in PK_4(G)$, and let for $k\in \N$
$$
\delta_k =\sup\big\{|\, f(xyz)+f(x)+f(y)+f(z)-f(xy)-f(xz)-f(yz)
\,|\, ; \,\, x, y, z\in S_k \big\}.
$$
Let us verify that $\delta_k =0$ for any $k$. Suppose that
$\delta_1 >0$. Then for some $x_1,y_1,z_1$ in $S_1$, we have
$$
|\,
f(x_1y_1z_1)+f(x_1)+f(y_1)+f(z_1)-f(x_1y_1)-f(x_1z_1)-f(y_1z_1)
\,|=\delta >0.
$$
By Lemma~\ref{lemma-2} there are $x_2,y_2,z_2\in S_2$ such that
$$
|\,
f(x_2y_2z_2)+f(x_2)+f(y_2)+f(z_2)-f(x_2y_2)-f(x_2z_2)-f(y_2z_2)
\,|=9\delta >0.
$$
By repeated applications of Lemma~\ref{lemma-2} we obtain, for any $k\in
\Bbb{N}$, there are $x_k,y_k,z_k\in S_k$ such that
$$
|\,
f(x_ky_kz_k)+f(x_k)+f(y_k)+f(z_k)-f(x_ky_k)-f(x_kz_k)-f(y_kz_k).
\,|=9^{k-1}\delta >0
$$
This gives  a contradiction to the assumption that $f\in PK_4(G)$.
Therefore $\delta_1=0$. Similarly,
using Lemma \ref{lemma-2a},
we verify that $\delta_n=0$
for any $n\in\N$. So, $PK_4(G)=K_4(G)$. Similarly we verify that
$PK_2(G)=K_2(G)$. Thus by Proposition~\ref{proposition3.3} we
get $PK(G)=K(G)$ and the equation~\eqref{kan} is stable on $G$.
This finishes the proof of the theorem.
\end{proof}

\medskip
\centerline{\small ACKNOWLEDGMENTS}
\medskip

The work was partially supported by an IRI Grant from the Office of
 the Vice President for Research, University of Louisville.

\end{document}